\documentclass{amsart}

\title{$R(5,5) \leq 46$}
\author{Vigleik Angeltveit and Brendan D. McKay}
\address[1]{Mathematical Sciences Institute \\
Australian National University \\
Canberra, ACT 2601 \\
Australia}
\address[2]{School of Computing \\
Australian National University \\
Canberra, ACT 2601 \\
Australia}
\email[1]{vigleik.angeltveit@anu.edu.au}
\email[2]{brendan.mckay@anu.edu.au \textrm{(corresponding author)}}

\usepackage{amsxtra}
\usepackage{amsfonts}
\usepackage[latin1]{inputenc}
\usepackage{graphicx,url}
\usepackage{amsmath,amssymb,latexsym,amsthm,mathrsfs}

\newtheorem{theorem}{Theorem}[section]
\newtheorem{thm}[theorem]{Theorem}
\newtheorem{lemma}[theorem]{Lemma}

\newtheorem{prop}[theorem]{Proposition}

\theoremstyle{definition}

\newtheorem{remark}[theorem]{Remark}

\makeatletter
\let\c@equation\c@theorem
\makeatother
\numberwithin{equation}{section}

 \newcommand{\cR}{\mathcal{R}}        

\newcommand{\kmin}{k_{\mathrm{min}}}
\newcommand{\kmax}{k_{\mathrm{max}}}
\newcommand{\emin}{e_{\mathrm{min}}}
\newcommand{\emax}{e_{\mathrm{max}}}

\renewcommand{\dfrac}[2]{\lower0.12ex\hbox{\large$\textstyle\frac{#1}{#2}$}}
\newcommand{\Dfrac}[2]{\raise0.05ex\hbox{\small$\displaystyle\frac{#1}{#2}$}}

\newcommand{\excess}{\operatorname{excess}}
\normalsize

\begin{document}

\begin{abstract}
We prove that the Ramsey number $R(5,5)$ is less than or equal to~$46$.
The proof uses a combination of linear programming and checking a large number of cases by computer.
All of the computational parts of the proof were independently implemented
by both authors, with consistent results.
\end{abstract}

\maketitle

\section{Introduction}
The Ramsey number $R(s,t)$ is defined to be the smallest $n$ such that every graph of order $n$ contains either a clique of $s$ vertices or an independent set of $t$ vertices. See \cite{Ra94} for a survey on the currently known bounds for small Ramsey numbers.

\begin{theorem} \label{t:main}
The Ramsey number $R(5,5)$ is less than or equal to $46$.
\end{theorem}

The lower bound of $43$, which was established by Exoo \cite{Ex89} in 1989, is still the best. The upper bound of $49$ was proved by the second author and Radziszowski \cite{McRa97}, and this was improved by the authors \cite{AnMc18} to $48$.

A \textit{Ramsey graph of type} $(s,t)$ is a simple graph with no clique of size $s$ or independent
set of size~$t$.
Let $\cR(s,t)$ denote the set of (isomorphism classes of) Ramsey graphs of type $(s,t)$, and let 
$\cR(s,t,n)$ denote the subset of those with $n$ vertices.
Let $\cR(s,t,n,e=e_0)$ denote the subset of $\cR(s,t,n)$ having $e_0$ edges, and similarly for $\cR(s,t,n,e \leq e_0)$ and $\cR(s,t,n,e \geq e_0)$. 
Let $e(s,t,n)$ and $E(s,t,n)$ denote the minimal and maximal number of edges of such a Ramsey graph, respectively.

The improvements in the upper bound for $R(5,5)$ are intimately connected to our understanding of $\cR(4,5,n)$ for $n$ large. Recall from \cite{McRa95} that $R(4,5) = 25$. It follows that any vertex $v$ of a graph in $\cR(5,5,m)$ must have degree $m-25 \leq d(v) \leq 24$.

It follows immediately that $R(5,5) \leq 50$, and that any graph in $\cR(5,5,49)$ must be regular of degree $24$. This, together with a partial census of the set $\cR(4,5,24)$, allowed the second author and Radziszowski \cite{McRa97} to prove that $R(5,5) \leq 49$. Their argument can be summarised as follows. First, they proved that $E(4,5,24) = 132$, and found the two graphs in $\cR(4,5,24,e=132)$.

Given a graph $F$ and a vertex $v \in V(F)$, let $F_v^+$ be the induced subgraph on the vertices adjacent to $v$ and let $F_v^-$ be the induced subgraph on the vertices not $v$ or adjacent to~$v$.
We call $F_v^+$ and $F_v^-$ the \textit{neighbourhood} and \textit{dual neighbourhood} of $v$ in~$F$, respectively.

Recall the following identity, e.g.\ from the $m=2$ case of \cite[Theorem 2.2]{McRa97},
which can be proved by counting the subgraphs induced by sets of three vertices.
If $F$ has $n$ vertices,
writing $e(\cdot)$ for the number of edges in its argument and
$d(v)$ for the degree of vertex~$v$, we have
\begin{equation} \label{e:edgeequation}
   \excess(F) =0, \text{~where~}
   \excess(F) = 
  \!\!\! \sum_{v\in V(F)}\! \bigl( e(F_v^-)-e(F_v^+) - \dfrac12d(v)(n-2d(v))\bigr).
\end{equation}
Given a hypothetical graph $F \in \cR(5,5,49)$, the above equation forces the neighbourhoods of all the vertices in $F$ to be one of those two graphs in $\cR(4,5,24,e=132)$ and the dual neighbourhoods to be one of the complement graphs.
The authors then used further subgraph identities to obtain a contradiction.

\begin{figure}[ht]
\[ \includegraphics[scale=0.6]{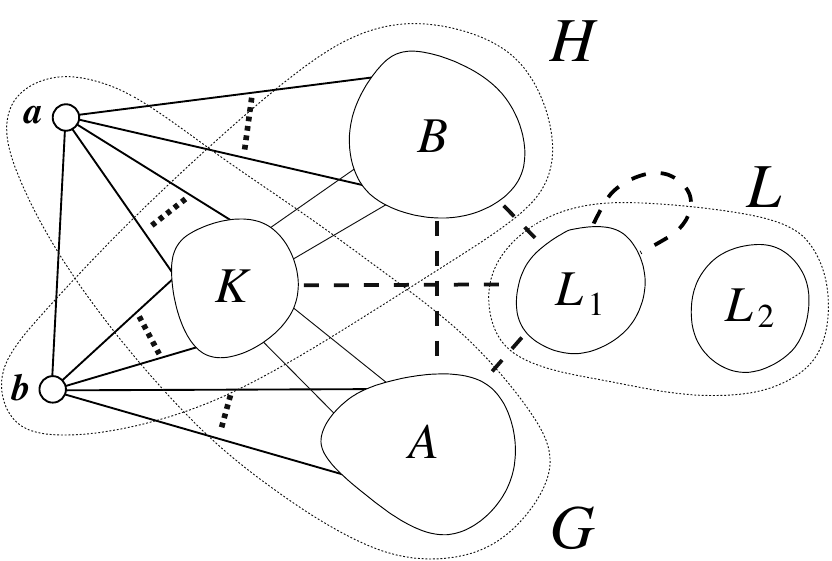} \]
\caption{The structure of a graph in $\cR(5,5,n)$.\label{fig:pic46}}
\end{figure}

\section{High level description of the method}\label{s:highlevel}

The method used in both \cite{AnMc18} and the present paper can be summarized with
the help of Figure~\ref{fig:pic46}, which shows a graph in $\cR(5,5,n)$, where $n=48$
in \cite{AnMc18} and $n=46$ here.
Consider two adjacent vertices $a,b$.  The neighbourhood of $a$ is $H=\{b\}\cup K\cup B$
and that of $b$ is $G=\{a\}\cup K\cup A$. Both of these neighbourhoods are in $\cR(4,5)$.
$L$ is the part of the graph adjacent to neither~$a$ nor~$b$.
The subgraph consisting of $G$ and $H$, overlapping in $K$, plus the edge $\{a,b\}$
but not the edges between $A$ and $B$, will be denoted by $G\cup_K H$.

Using a mixture of theory and computation we compile a collection of pairs
$\{(G,a), (H,b)\}$ such that $G_a^+$ is isomorphic to $H_b^+$ and
every graph in $\cR(5,5,n)$ necessarily contains a pair in our collection
overlapped as in the figure with $a$ adjacent to~$b$.
We call a pair like $(G,a)$ a \textit{pointed graph}.
Next, for each pair of pointed graphs in our collection, we determine all the ways
to fill in edges in the places indicated by heavy dashed lines in the figure
without creating cliques or independent sets of size~5.
We call this \textit{gluing along the edge $ab$}.
The choice of $L_1\subseteq L$ is arbitrary so we may choose any $L_1$ that
is necessarily present if this structure extends to a graph in $\cR(5,5,n)$.

In our previous work \cite{AnMc18}, our strategy was to start with $L_1=\emptyset$, so that only
edges between $A$ and $B$ were sought.
This produced a large but manageable number of solutions.
Then we attempted to extend each solution to a graph in $\cR(5,5)$ by adding one
additional vertex.
Since no such extension existed in any of the cases, we concluded that $\cR(5,5,48)=\emptyset$.

With $n=46$, choosing $L_1=\emptyset$ is impractical as the number of valid ways
to add edges between $A$ and $B$ is extremely large.
Instead, we chose a larger $L_1$ and found that the number of solutions
became manageable.
The method for choosing $L_1$ differed between the approaches of the two
authors, and will be described in Sections~\ref{s:first} and~\ref{s:second}.

\medskip
We next outline the method for choosing a collection of pairs $\{(G,a), (H,b)\}$ of pointed graphs.

The proof that $R(5,5) \leq 48$ in \cite{AnMc18} can be summarised as follows. First we determined the complete catalogue of $\cR(4,5,24)$, which contains a total of $352{,}366$ graphs. Given a hypothetical graph $F \in \cR(5,5,48)$, either $F$ or its complement must have a pair of adjacent vertices $a, b$ of degree $24$ whose neighbourhoods intersect in some subgraph $K \in \cR(3,5,d)$ for $d \leq 11$.
(The fact that it is possible to choose $d \leq 11$ requires a proof, see \cite{AnMc18}.) Let $G = F_b^+$ and $H = F_a^+$ be the neighbourhoods of $b$ and $a$ in $F$. Then $K=G_a^+ = H_b^+$, and a large subgraph of $F$ can be reconstructed as $G \cup_K H$. Hence it suffices to consider all ways of gluing $G \cup_K H$ for $K \in \cR(3,5,d)$ where $d \leq 11$ and $G, H \in \cR(4,5,24)$. 
There is one gluing operation required for each pair of pointed graphs of type $K$ and for each automorphism of $K$, and in total we computed approximately $2$ trillion gluing operations. As mentioned above we will call this gluing along an edge. Note that \cite{AnMc18} relied on the complete catalogue of $\cR(4,5,24)$ but did not use any linear programming.

In the current paper we take the ideas from \cite{McRa97} and \cite{AnMc18} much further. For a hypothetical graph $F \in \cR(5,5,46)$ every vertex must have degree $d(v) \in \{21,22,23,24\}$, and we consider parts of $\cR(4,5,n)$ for $n=21,22,23,24$. From \cite{McRa95} we know that $\cR(4,5,n)$ is quite large for $n=21,22,23$. Indeed, in \cite{McRa95} the authors estimated that $|\cR(4,5,21)| \approx 5.5 \times 10^{17}$, $|\cR(4,5,22)| \approx 1.9 \times 10^{15}$ and $|\cR(4,5,23)| \approx 10^{11}$.
(See the Appendix for updated estimates.)
 Hence any approach that relies on the complete catalogue of these Ramsey graphs is impractical.

Instead we use linear programming to reduce the number of graphs we need to consider. The basic idea is to determine $\cR(4,5,n, e \geq e_0)$ for $n=21, 22, 23$ and suitable $e_0$ (which depends on $n$) and exclude these graphs by gluing along an edge as described above. Then we can use linear programming to finish the proof of Theorem \ref{t:main}.

In more detail, we determine the sets $\cR(4,5,23, e \geq 119)$, $\cR(4,5,22, e \geq 113)$ and $\cR(4,5,21, e \geq 107)$. We also use $\cR(4,5,24, e \geq 127)$, determined in \cite{AnMc18}.
Then we glue these graphs along an edge. There are some additional challenges:

\begin{enumerate}
 \item The method used in \cite{McRa95, AnMc18} to determine $\cR(4,5,24)$ is too slow to determine $\cR(4,5,n)$ for $n=21, 22, 23$, so we cannot simply compute all of $\cR(4,5,n)$ and throw away the graphs with too few edges. We explain our approach in Section \ref{s:census}.
 \item The method used in \cite{AnMc18} to glue two graphs along an edge produces far too many output graphs, and even if we could collect them all it would take too long to compute all possible ways to add one vertex while staying within $\cR(5,5)$. We get around that by adding extra vertices straight away when gluing along an edge as in \cite{AnMc18}. 
 \end{enumerate}

In addition, we show in Section \ref{s:gluepairs} that we do not have to perform all possible gluing operations. In fact, we show that we can leave out some of the more time-consuming gluing operations.

We estimate that completing the census of $\cR(4,5,n, e \geq e_0)$ took approximately $15$ years of CPU time while gluing the necessary graph took another $15$ years of CPU time. Hence the whole project took about $30$ years of CPU time for the first author to complete.

In the interests of confidence, all the computations were repeated by the second author using independent programs and usually with different methods.
This replication took about $50$ years of additional CPU time.
For some tasks, such as creation of the catalogues described in the following section,
the outputs of the two implementations could be directly compared.
For the remaining much more complicated stages, the two approaches were deliberately
intended to be disjoint to 
avoid the well-known axiom of software engineering that two 
programmers implementing the same algorithm tend to make errors in the same places.
Moreover, if a graph $\cR(5,5,46)$ exists, each implementation should have found many of its subgraphs by their completely different approaches.
We feel that this provides better confidence than two implementations of the
same approach would provide.

\section{Graphs in $\cR(4,5,n)$ with many edges} \label{s:census}
Recall that in \cite{AnMc18} the authors completed the census of graphs in $\cR(4,5,24)$ that was started in \cite{McRa95}, and that $|\cR(4,5,24)|=352{,}366$. To proceed, we also need to know something about the graphs in $\cR(4,5,n)$ for $n=21, 22, 23$.
As mentioned in the introduction there are far too many such graphs, see Table~\ref{tab1}.

Fortunately we do not need all of the graphs in $\cR(4,5,n)$ but only the ones with a large number of edges. We consider $\cR(4,5,n)$ for $n=24, 23, 22, 21$ in turn. Note that different choices for how many graphs of each type to consider are possible. For example, we could consider $\cR(4,5,22,e=114)$ only at the price of having to consider $\cR(4,5,23,e \geq 118)$ instead of $\cR(4,5,23,e \geq 119)$.
It also seems like it would suffice to consider $\cR(4,5,24,e \geq 128)$ rather than $\cR(4,5,24,e \geq 127)$. But if we did that we would have to weaken Proposition \ref{p:gluered} and perform more of the more difficult gluing operations.

\subsection{$\cR(4,5,24)$}
We have $E(4,5,24) = 132$, and for the proof of Theorem \ref{t:main} it suffices to consider $\cR(4,5,24, e \geq 127)$. We recall from \cite{AnMc18} that we have
\begin{align*}
 |\cR(4,5,24,e=132)| & =  2 \\
 |\cR(4,5,24,e=131)| & =  3 \\
 |\cR(4,5,24,e=130)| & =  32 \\ 
 |\cR(4,5,24,e=129)| & =  147 \\ 
 |\cR(4,5,24,e=128)| & =  843 \\ 
 |\cR(4,5,24,e=127)| & =  3{,}401
\end{align*}

\subsection{$\cR(4,5,23)$}
We have $E(4,5,23) = 122$, and for the proof of Theorem \ref{t:main} it suffices to consider $\cR(4,5,23, e \geq 119)$. We have
\begin{align*}
 |\cR(4,5,23,e=122)| & =  2 \\
 |\cR(4,5,23,e=121)| & =  119 \\
 |\cR(4,5,23,e=120)| & =  7{,}800 \\ 
 |\cR(4,5,23,e=119)| & =  332{,}778
\end{align*}
Finding all of these graphs was the second most time consuming part of the project, taking approximately 5 years of CPU time.

To find these graphs, we used the following strategy: As in \cite{McRa95}, the idea is to glue $\cR(3,5,p)$ to $\cR(4,4,q)$ for $p+q+1=23$. Given $G \in \cR(3,5,p)$ and $H \in \cR(4,4,q)$, denote the vertices of $H$ by $v_0,\ldots,v_{q-1}$. Now consider all tuples $(d_0,\ldots,d_{q-1})$ so that if we glue $G$ to $H$ with $v_i$ adjacent to $d_i$ vertices in $G$, the results lie in $\cR(4,5,23,e \geq 119)$.

Now we can order the vertices of $H$ in a clever way, balancing two objectives: We want a vertex $v_i$ with $d_i$ large near the beginning, and we want a dense subgraph of $H$ near the beginning. We can do this for all $H \in \cR(4,4,q)$ and all tuples $(d_0,\ldots,d_{q-1})$ with sufficiently large sum, and organise the set of such pairs $(H, (d_0,\ldots,d_{q-1}))$ into a tree. The advantage of doing this is that after attaching, say, $k$ vertices to $G$ we produce graphs in $\cR(4,5,p+k+1, e \textnormal{ large})$, and this produces a bottleneck that we can take advantage of.

\subsection{$\cR(4,5,22)$}
We have $E(4,5,22) = 114$, and for the proof of Theorem \ref{t:main} it suffices to consider $\cR(4,5,22, e \geq 113)$. We have
\begin{align*}
 |\cR(4,5,22,e=114)| & = 133 \\
 |\cR(4,5,22,e=113)| & = 30{,}976
\end{align*}
We used a similar strategy as above. This calculation was significantly faster than the one for $\cR(4,5,23,e \geq 119)$.

\subsection{$\cR(4,5,21)$}
We have $E(4,5,21) = 107$, and for the proof of Theorem \ref{t:main} it suffices to consider $\cR(4,5,21, e=107)$. We have
\[
 |\cR(4,5,21,e=107)| = 31.
\]
Once again we used a similar strategy as above, and this calculation was faster still.

Some additional counts that we don't need for this project are given
in Table~\ref{tab1}, and the graphs themselves, including additional classes,
are available on the internet~\cite{R45web}. 

\section{Some linear programming}
The purpose of this section and the next is to define a set of gluing
operations sufficient to prove Theorem~\ref{t:main}.
This set will be given in Theorem~\ref{t:gluered}.

Define the following sets:
\begin{align*}
 A & =  \cR(4,5,24, e \geq 127) \\
 B_1 & =  \cR(4,5,23, e \geq 121) & B_2 &= \cR(4,5,23, e = 120) & B_3 &= \cR(4,5,23, e=119) \\
 C_2 & =  \cR(4,5,22, e = 114) & C_3 &= \cR(4,5,21, e = 113) \\
 D_3 & =  \cR(4,5,21, e = 107)
\end{align*}

Define $\bar{A}$ to be the set of complements of the members of $A$, and so on.
So for example, $\bar{A}  = \cR(5,4,24, e \leq 149)$.
Let
\[
 E = A \cup B_1 \cup B_2 \cup B_3 \cup C_2 \cup C_3 \cup D_3.
\]

Now consider the special case of a graph $F \in \cR(5,5,46)$. Then every vertex of $F$ has degree in $\{21,22,23,24\}$ and we can write~\eqref{e:edgeequation} as
\begin{align*}
 \textnormal{excess}(F) & =  \sum_{d(v) = 24} \bigl( e(F_v^-) - e(F_v^+) + 24 \bigr) \\
 & {\quad}+  \sum_{d(v) = 23} \bigl( e(F_v^-) - e(F_v^+) \bigr) \\
 & {\quad}+  \sum_{d(v) = 22} \bigl( e(F_v^-) - e(F_v^+) - 22 \bigr) \\
 & {\quad}+  \sum_{d(v) = 21} \bigl( e(F_v^-) - e(F_v^+) - 42 \bigr) = 0.
\end{align*}

On the first line, with $d(v) = 24$, $F_v^- \in \cR(5,4, 21)$ while $F_v^+ \in \cR(4,5,24)$, and so on. We can rewrite $\excess(F)$ as
\begin{align*}
 \textnormal{excess}(F) & =  \sum_{d(v) = 24} \bigl( (e(F_v^-) - 104) + (127 - e(F_v^+)) + 1 \bigr) \\
 & {\quad}+  \sum_{d(v) = 23} \bigl( (e(F_v^-) - 119) + (118 - e(F_v^+)) + 1 \bigr) \\
 & {\quad}+  \sum_{d(v) = 22} \bigl( (e(F_v^-) - 135) + (112 - e(F_v^+)) + 1 \bigr) \\
 & {\quad}+  \sum_{d(v) = 21} \bigl( (e(F_v^-) - 149) + (106 - e(F_v^+)) + 1 \bigr).
\end{align*}

Given such an $F \in \cR(5,5,46)$, suppose each $F_v^+ \not \in E$ and each $F_v^- \not \in \bar{E}$. Then each vertex $v \in V(F)$ contributes at least $1$ to $\excess(F)$, so $\excess(F) \geq 46$ and hence we cannot have $\textnormal{excess}(F) = 0$. This suggests a proof strategy: Deal with the relatively small number of graphs in $E$ (and $\bar{E}$) separately, and then use the above equation for $\textnormal{excess}(F)$ to finish.

\begin{remark}
As alluded to above, the above argument works with $\cR(4,5,24,e \geq 128)$ instead of $\cR(4,5,24, e \geq 127)$. But we include $e=127$ because otherwise we cannot prove Theorem \ref{t:gluered}.
\end{remark}

\section{Gluing pairs of pointed graphs} \label{s:gluepairs}
The obvious strategy is to consider all pointed graphs constructed from $E$, and glue every pair of pointed graphs of type $K$ for each $K \in \cR(3,5,d)$. If we can do this, it will imply that there are at most $4$ vertices of $F \in \cR(5,5,46)$ with neighbourhoods in $E$ because with $5$ such vertices there must be an edge between two of them and gluing along that edge will detect $F$. Similarly this implies that there are at most $4$ vertices with dual neighbourhoods in $\bar{E}$. Each $F_v^+$ or $F_v^-$ not in $E$ decreases the estimate of $\excess(F)$ by at most $5$, and we get $\excess(F) \geq 46 - 8 \cdot 5 = 6$.

There are two problems with this strategy:

\noindent
\textbf{Problem 1:}
If we glue pointed graphs $(G,a) \in \cR(4,5,n_1, K)$ and $(H,b) \in \cR(4,5,n_2,K)$ for $K \in \cR(3,5,d)$ along an edge the output is a collection of graphs in 
$\cR(5,5,n_1+n_2-d)$. This worked well when $n_1=n_2=24$ and $d \leq 11$, as this produced graphs in $\cR(5,5, n \geq 37)$ and there are not very many of those. (Well, actually there are lots of such graphs. But there are very few such graphs that also have two vertices of degree $24$.) We solve that by including additional vertices, as indicated in Figure \ref{fig:pic46}.

\noindent
\textbf{Problem 2:}
As we decrease $n_1$ and $n_2$, and increase $d$, the gluing calculations take a lot longer. This should not be surprising, as we are starting with fewer specified edges. We put a lot of work into optimising the gluing program, but even so the calculation would have taken too long. We solve this by showing that it suffices to glue only some of the pairs of pointed graphs from $E$.

Define the following sets of graphs:
\begin{align*}
 E_1 & =  A \cup B_1, \\
 E_2 & =  B_2 \cup C_2, \\
 E_3 & =  B_3 \cup C_3 \cup D_3.
\end{align*}

Define $\bar{E}_1,\bar{E}_2,\bar{E}_3$ to be the sets of complementary graphs.
Given a graph $G \in \cR(4,5,n)$, let $P(G)$ denote the corresponding set of pointed graphs. This set usually consists of $n$ pointed graphs, although some pointed graphs might be isomorphic.
Let $P_k(G)$ denote the set of pointed graphs obtained as follows:
First, find the $n$ pointed graphs for $G$ (without removing isomorphic graphs).
Then sort them according to a heuristic measure of difficulty, and throw away the $k-1$ most difficult pointed graphs. At this point we can throw away isomorphic copies of the remaining pointed graphs.

Given $F \in \cR(5,5,46)$, we can consider the induced subgraph $F[E]$ of $F$ on the vertices with neighbourhood in $E$. Similarly, let $F[\bar{E}]$ denote the induced subgraph of $F$ on the vertices with dual neighbourhood in $\bar{E}$.

In the following, we abuse notation by saying that a vertex $v \in V(F)$ is in $E$ if its neighbourhood is in $E$, and similarly for $E_i$.

\begin{lemma} \label{l:R5417deg8}
Any graph $G \in \cR(5,4,17)$ has a vertex of degree at least $8$.
\end{lemma}

\begin{proof}
This is an explicit calculation. If $G$ has at least $60$ edges then this is automatic, and we can check this by completing a census of $\cR(5,4,17, e \leq 59)$. There are 7147 such graphs, and all of them have at least one vertex of degree greater than or equal to $8$.
\end{proof}

\begin{lemma} \label{l:R5521deg8}
Suppose $G \in \cR(5,5,21)$ has two non-adjacent vertices of degree at most $4$. Then either $G$ contains a vertex of degree at least $8$ or $G$ contains a $4$-clique $\{w_1,w_2,w_3,w_4\}$ with $\deg_G(w_1)+\deg_G(w_2)+\deg_G(w_3)+\deg_G(w_4) \leq 24$.
\end{lemma}

\begin{proof}
This is another explicit calculation. If $G$ has a vertex of degree $3$ then the result follows from Lemma \ref{l:R5417deg8} by considering its dual neighbourhood. If $G$ has two non-adjacent vertices $\{v_1,v_2\}$ of degree $4$ then the dual neighbourhood of $v_1$ is of type $\cR(5,4,16)$ and has a vertex of degree at most $4$. By an explicit calculation there are $2029$ graphs in $\cR(5,4,16)$ with maximum degree at most $7$ and a vertex of degree at most $4$. After adding $v_1$ and $4$ more vertices we find that there are $148$ graphs in $\cR(5,5,21)$ with maximum degree at most $7$ and two non-adjacent vertices of degree at most $4$. And all of these have a $4$-clique satisfying the condition in the lemma.
\end{proof}

\begin{prop} \label{p:gluered}
Given $F \in \cR(5,5,46)$, either $F[E]$ or the complement $\bar{F}[E]$ must contain one of the following:
\begin{enumerate}
 \item Some vertex in $E_1$ adjacent to at least $1$ other vertex in $E$. \label{glue1}
 \item Some vertex in $E_2$ adjacent to at least $1$ other vertex in $E_2$. \label{glue2}
 \item Some vertex in $E_2$ adjacent to at least $5$ other vertices in $E$. \label{glue3}
 \item Some vertex in $E_3 \setminus D_3$ adjacent to at least $8$ other vertices in $E$. \label{glue4}
 \item Some vertex in $D_3$ adjacent to at least $8$ other vertices in $E \setminus D_3$. \label{glue5}
\end{enumerate}
\end{prop}

\begin{proof}
Suppose $F \in \cR(5,5,46)$ has $m_i$ vertices with neighbourhood in $E_i$ and $\bar{m}_i$ vertices with dual neighbourhood in $\bar{E}_i$ for $i=1, 2, 3$. Let $\alpha = 5m_1+2m_2+m_3$ and $\beta=5\bar{m}_1+2\bar{m}_2+\bar{m}_3$. Then $\excess(E) \geq 46  - \alpha - \beta$, so since $\excess(E) = 0$ we get $\alpha + \beta \geq 46$.

Moreover, let $n_{21}$ be the number of vertices in $F[E]$ of degree $21$ in $F$ and define $\bar{n}_{21}$ similarly using $\bar{F}$. Then, if $n_{21} > \bar{m}_1$ we have at least $n_{21}-\bar{m}_1$ vertices in $F$ with dual neighbourhoods in $\cR(5,4,24,e\geq 150)$ and together with the dual consideration it follows that $\alpha + \beta \geq 46 + \max(n_{21}-\bar{m}_1, 0) + \max(\bar{n}_{21}-m_1, 0)$. (This was the reason for considering $\cR(4,5,24, e \geq 127)$ rather than $\cR(4,5,24, e \geq 128)$.)

If $F$ does not satisfy the conclusion of the proposition then the induced subgraph $F[E]$ is in $\cR(5,5,m_1+m_2+m_3)$. Moreover, the vertices in $E_1$ have degree $0$ in $F[E]$, so $m_1 \leq 4$ and the induced subgraph on the remaining vertices is in $\cR(5, 5-m_1, m_2+m_3)$. The vertices in $E_2$ have degree at most $4$ in $F[E]$ and are only adjacent to vertices in $E_3$. Finally, the vertices in $E_3 \setminus D_3$ have degree at most $7$ and the vertices in $D_3$ have degree at most $7 + (n_{21}-1)$.

First we claim that if $\alpha \geq 21$ then $m_1 \leq 2$. To prove this, we do a case by case analysis. Suppose this fails for some $F \in \cR(5,5,46)$. If $m_1=4$ then $m_2+m_3 > 0$, and we get an independent $5$-set on the $4$ vertices in $E_1$ and one additional vertex in $E$, a contradiction. If $m_1=3$ and $m_2 \geq 2$ we again get an independent $5$-set. If $m_1=3$ and $m_2 \leq 1$ then $m_2+m_3 \geq 5$. If all the vertices in $E_2 \cup E_3$ are adjacent we get a $5$-clique. Otherwise we get an independent $5$-set.

Second, we claim that if $\alpha \geq 23$ then either $\alpha=23$, $m_1=2$ and $n_{21} = 13$ or $m_1 \leq 1$ and $n_{21} \geq \alpha - 21$. To prove this, we first consider the case $m_1 = 2$. If $m_1=2$ we must have $m_2 \leq 2$ and $2m_2+m_3 \geq 13$. If $m_2 > 0$ then the vertices in $E_2 \cup E_3$ form a graph in $\cR(5,3,d\geq 11)$. Every vertex in such a graph has degree at least $6$, so $F[E]$ satisfies (\ref{glue3}). If $m_2=0$ then $m_3\geq \alpha-10$ and the vertices in $E_3$ form a graph in $\cR(5,3,\alpha-10)$. If $\alpha > 23$ then $\cR(5,3,\alpha-10)$ is empty. If $\alpha=23$ then every vertex in such a graph has degree at $8$, so either $F[E]$ satisfies (\ref{glue4}) or all $13$ vertices in $\cR(5,3,13)$ are in $D_3$.

To finish proving the claim it suffices to show that if $\alpha = 23$ and $n_{21} \leq 1$ then $F$ satisfies the conclusion of the proposition. (If $\alpha > 23$ and $n_{21} \leq \alpha-22$, simply remove $\alpha-23$ vertices from $E$, starting with vertices of degree $21$ in $F$.) Suppose not. It suffices to consider $m_1=1$ and $m_1 = 0$. We do each case in turn.

If $\alpha=23$, $m_1 = 1$ and $n_{21} \leq 1$ then either $F[E]$ satisfies (\ref{glue1}) or (\ref{glue2}), or $m_2 \leq 3$ and $2m_2+m_3 \geq 18$. The vertices in $E_2 \cup E_3$ form a graph in $\cR(5,4,m_2+m_3)$, and either $F[E]$ satisfies (\ref{glue3}) or the vertices in $E_2$ have degree at most $4$. If $m_2 \leq 1$ then we are done by Lemma \ref{l:R5417deg8}. If $m_2 \geq 2$ then the dual neighbourhood of the first vertex in $E_2$ is in $\cR(5,3,d\geq 10)$ and since every vertex in such a graph has degree at least $5$ we conclude that $F[E]$ satisfies (\ref{glue3}).

If $\alpha=23$, $m_1=0$ and $n_{21} \leq 1$ then either $F[E]$ satisfies (\ref{glue1}) or (\ref{glue2}), or $m_2 \leq 4$ and $2m_2+m_3 \geq 23$. Now we consider each value of $m_2$ separately. If $m_2=3$ or $m_2=4$ then we can pick two vertices in $E_2$ in such a way that the intersection of their dual neighbourhoods form a graph in $\cR(5,3,d \geq 10)$. This follows from an application of the inclusion-exclusion principle. If $m_2=4$, let $\{w_1,w_2,w_3,w_4\}$ be the vertices in $E_2$ and let $A_i$ for $i=1,2,3,4$ be the vertices in $E_3$ not adjacent to $w_i$. Since each $w_i$ has degree at most $4$ in $F[E]$, each $A_i$ has size at least $11$. Let us write $A_{ij}$ for $A_i \cap A_j$, and so on. If the conclusion fails, each double intersection has size at most $7$. Now we use that $\sum |A_i| = 2 \sum |A_{ij}| - 3 \sum |A_{ijk}|$ to conclude that
\[
 \Bigl|\bigcup A_i\Bigr| = \sum |A_i| - \sum |A_{ij}| + \sum |A_{ijk}| \geq \dfrac{2}{3} \sum |A_i| - \dfrac{1}{3} \sum |A_{ij}| \geq 15 + \dfrac{1}{3},
\]
a contradiction since $\bigl|\bigcup A_i\bigr| = 19 - 4 = 15$. The case $m_2=3$ is similar. Since every vertex in such a graph has degree at least $5$, it follows that $F$ satisfies (\ref{glue2}) or (\ref{glue3}).

If $m_2=2$ then $m_3 \geq 19$, and $F[E]$ (or an induced subgraph) lies in $\cR(5,5,21)$. If $F$ does not satisfy the conclusion of the proposition then $F[E] \in \cR(5,5,21)$ has two non-adjacent vertices of degree at most $4$, and maximal degree at most $7$. Now there are some such graphs, so further analysis is required. But we can find all of them, starting from $\cR(5,4,16,e \leq 52)$. Given such a graph $G$, we can consider all $4$-cliques and use an inclusion-exclusion argument. This gives a contradiction as follows: By Lemma \ref{l:R5521deg8} there is a $4$-clique $\{w_1,w_2,w_3,w_4\}$ in $G$ with $\sum \deg_G(w_i) \leq 24$. Let $A_i$ be the neighbours of $w_i$ in $F[V(F)-E]$. Since $3$ of $\{w_1,w_2,w_3,w_4\}$ have degree at least $22$ in $F$ and the last $w_i$ have degree at least $21$ in $F$, we get $\sum |A_i| \geq 3 \cdot 22 + 21 - 24 = 63$. Now we use that $\sum |A_i| = 2 \sum |A_{ij}| - 3 \sum |A_{ijk}|$ to conclude that
\[
 \Bigl|\bigcup A_i\Bigr| = \sum |A_i| - \sum |A_{ij}| + \sum |A_{ijk}| \geq \dfrac{63}{2} - \dfrac{1}{2} \sum |A_{ijk}| \geq 25 + \dfrac{1}{2}.
\]
This yields the desired contradiction, as $\bigl|\bigcup A_i\bigr| \leq 46-21=25$.

If $m_2=1$ then $m_3 \geq 21$ and $F[E]$ (or an induced subgraph) lies in $\cR(5,5,22)$. If $F$ does not satisfy the conclusion of the proposition then $G \in \cR(5,5,22)$ has one vertex of degree at most $4$. The dual neighbourhood of this vertex is in $\cR(5,4,d\geq 17)$, and now we use Lemma \ref{l:R5417deg8} to conclude.

Finally, if $m_2=0$ then $m_3 \geq 23$ and $F[E]$ (or an induced subgraph) lies in $\cR(5,5,23)$. If $F$ does not satisfy the conclusion of the proposition then $G \in \cR(5,5,23)$, and $G$ has maximal degree at most $7$. Such a $G$ must have a $4$-clique on vertices $\{w_1,w_2,w_3,w_4\}$ with each $w_i$ having degree at least $22$ in $F$. Then each $w_i$ is adjacent to at least $15$ vertices in $V(F)-V(G)$, and we obtain a contradiction by using the inclusion-exclusion principle in the same way as above, now with
\[
 \Bigl|\bigcup A_i\Bigr| \geq \dfrac{60}{2} - \dfrac{1}{2} \sum |A_{ijk}| = 24.
\]

Now we can finish the proof. Without loss of generality we can assume $\alpha \geq \beta$, and we immediately get $\alpha \geq 23$. If $\beta \leq 20$ then $\alpha \geq 26 + (\alpha - 21 - \bar{m}_1)$, which implies $\bar{m}_1 \geq 5$, a contradiction. If $\beta = 21$ or $\beta = 22$ we get $\alpha \geq 24 + (\alpha-21-\bar{m}_1)$, which implies $\bar{m}_1 \geq 3$, a contradiction. And if $\beta \geq 23$ we get $\alpha + \beta \geq 46 + (\alpha-21-\bar{m}_1) + (\beta-21-m_1)$, so $m_1 + \bar{m}_1 \geq 4$. If $m_1 \geq 3$ or $\bar{m}_1 \geq 3$ we are done. Otherwise $(\alpha, m_1, n_{21}) = (23, 2, 13)$ and $(\beta, \bar{m}_1, \bar{n}_{21}) = (23,2,13)$ and again we are done.
\end{proof}

Given $F$, it suffices to glue along a single edge in $F[E]$.
Hence it follows from Proposition \ref{p:gluered} that it suffices to do the following gluings:

\begin{thm} \label{t:gluered}
To prove Theorem \ref{t:main} it suffices to do the following gluings:
\begin{enumerate}
 \item Glue $P(E_1)$ to $P(E)$;
 \item Glue $P(E_2)$ to $P(E_2)$;
 \item Glue $P_5(E_2)$ to the rest of $P(E)$;
 \item Glue $P_8(E_3 \setminus D_3)$ to $P(E)$.
 \item Glue $P_8(D_3)$ to $P(E \setminus D_3)$.
\end{enumerate}
\end{thm}

\section{First gluing computation}\label{s:first}
In this section we describe the method used by the first author. It follows the description in Section \ref{s:highlevel}, with $G$, $H$ chosen from the pairs listed in Theorem \ref{t:gluered}. Each gluing problem can be encoded as a SAT problem whose clauses forbid cliques or independent sets of size $5$. We also added symmetry breaking clauses to distinguish the vertices of $L_1$.

If $|G| + |H| - |K| \geq 37$ we chose $L_1$ to be empty, and if $|G| + |H| - |K| < 37$ we chose $L_1$ so that the output graphs (if any) have exactly 37 vertices.

We are entitled to make some additional assumptions about $L_1$, and some experimentation suggested the following: Pick a vertex $v$ in $K$ of maximal degree $d$ in $G \cup_K H$, and choose $L_1$ so that it is adjacent to $\min(|L_1|, 21-d)$ of the vertices in $L_1$. In other words, we try to choose $v$ and $L_1$ so that $v$ has degree $21$. We can do that because any vertex of $F \in \cR(5,5,46)$ has degree at least $21$.

When running our solver, we then prioritised branching on variables corresponding to edges in the neighbourhood of $v$. If $v$ has degree $21$ in $G \cup_K H \cup L_1$ then the neighbourhood of $v$ has to be in $\cR(4,5,21)$, and because $\cR(4,5,21)$ is at least somewhat smaller than $\cR(4,5,19)$ or $\cR(4,5,20)$ (see the appendix) this provided an additional bottleneck in the calculation.

These gluing operations produced a total of 8,485,247 graphs in $\cR(5,5,37)$. None of those extended to $\cR(5,5,38)$.

Our SAT solver was a very simple special purpose solver without advanced features like clause learning and restarts.

\section{Second gluing computation}\label{s:second}

In this section we describe the method used by the second author.
It follows the description in Section~\ref{s:highlevel}, with
$G,H$ chosen from the pairs listed in Theorem~\ref{t:gluered}.

Let $v$ be a vertex in $K$ whose degree $\kmin$ is the least
so far (i.e. given only its connections to $a,b,A,B,K$).
Let $w$ be a vertex in $K$ whose degree $\kmax$ is the greatest
so far that is less than~21.  Possibly $w$ doesn't exist, in which case
references to it should be ignored.

The flexibility in choosing $L_1$ allows us to make some assumptions which
give the SAT solving a head start and also remove some of the symmetry.
\begin{itemize}
   \item[$\bullet$] $|L_1|=1$ :  If $w$ exists, it is adjacent to $L$ so we can
     take $L_1$ adjacent to~$w$.
   \item[$\bullet$] $|L_1|=2$ and $|L|\ge 4$ :  If $\kmin \le 17$, $v$ is adjacent
     to at least 4 vertices of~$L$.  Those 4 vertices must include an edge or else
     they form an independent set of size 5 in conjunction with~$a$.
     We can take $L_1$ to be that edge, with both ends adjacent to~$v$.
     If $18\le\kmin\le 19$, we can similarly take $L_1$ to be an edge with
     at least one end adjacent to~$v$.
   \item[$\bullet$] $|L_1|=3$ and $|L|\ge 9$ :  Since $R(3,4)=9$, $L$ contains
     a triangle which we can take as $L_1$.  If $\kmin \le 14$ we can also add one
     edge between $v$ and $L_1$.
   \item[$\bullet$] $|L_1|=4$ and $|L|\ge 13$ :  By direct computation we find that
     every graph in $\cR(5,4,13)$ contains an induced subgraph consisting of
     two triangles sharing an edge.  So we can assume $L_1$ contains
     this graph of five edges and one non-edge.
   \item[$\bullet$] $|L_1|=5$ and $|L|\ge 14$ :  By direct computation, every graph
     in $\cR(5,4,14)$ contains either $B_1$ or $B_2$ as an induced subgraph.
      \[ 
        \includegraphics[scale=0.45]{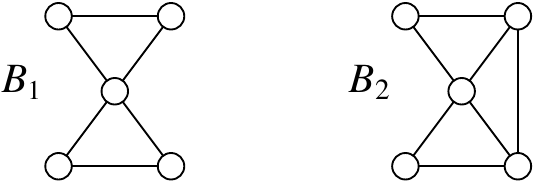} 
      \]
     If either $G$ or $H$ has 21 vertices, its complementary neighbourhood belongs
     to $\cR(5,4,24)$.  By direct computation, we find that none of the 413 graphs in
     $\cR(5,4,14)$ without $B_2$ as an induced subgraph extend to $\cR(5,4,24)$.
     Therefore, we can take $L_1$ to contain all the edges of $B_1$ and the 
     non-edges of $B_2$.  The remaining vertex pair in $L_1$ can be assumed
     to be an edge if either $G$ or $H$ has 21 vertices, and is left unspecified
     otherwise.
\end{itemize}

The computation now proceeded in phases.
Initially, $|L_2|$ was chosen to be 10, but the program automatically switched
to $|L_2|=9$ if too many solutions were being found.
Each phase consisted of these steps:
\begin{itemize}
\item[(1)]   The SAT clauses were propagated and a limited amount of backtracking
   was performed if the propagation was inconclusive.  This lead to either
   elimination, a solution found, or an inconclusive result.
   No attempt was made to find all solutions.
\item[(2)] If satisfiability was still unsettled, the reduced
   SAT program was handed to the SAT solver Glucose~\cite{Glucose} with a
   time limit of 1 minute. Glucose decided satisfiability in 1 second on average.
  
\item[(3)] If satisfiability was still unsettled (a very rare event), a combination of
   Glucose and backtracking was applied.  All such cases showed infeasibility.
\end{itemize}

If a configuration produced a solution, it was passed to the next phase with
$|L_1|$ increased by~1.  Each phase had many fewer configurations to process
than the previous phase.
As a sanity check, a number of cases were run as well using Kissat~\cite{Kissat}
instead of Glucose.

In the first phase, the total number of $(G,H)$-overlaps was about 12 million,
of which 5.6 million were processed by Glucose.
Glucose found 15,248 satisfiable cases, and the remainder were unsatisfiable
including 77 which timed out and needed step~(3).
The 15,248 satisfiable cases went to the second phase, where only 17
cases were found to be satisfiable and were sent to the third phase
where they were unsatisfiable.

This proves that none of the initial configurations can be extended
to $\cR(5,5,46)$.

\section{Conclusions}

Given the theory and the two independent computations, 
Theorem \ref{t:main} has now been firmly established.
We believe that future improvements to the upper bound on $R(5,5)$
will need new theoretical insights, as an excessive amount of
computer time would be required to apply the same method.
Improved upper bounds on some Ramsey numbers that partly rely
on our investigation of $\cR(4,5)$ can be found in the 2024 edition
of Radziszowki's survey~\cite{Ra94}.

\appendix
\section{A census of $\cR(4,5)$}

While only a small part of $\cR(4,5)$ was required for our main theorem,
we record here an update of \cite[Table 3]{McRa95}.
Recall that $e(4,5,n)$ and $E(4,5,n)$ are, respectively, the minimum and
maximum number of edges in $\cR(4,5,n)$.
In Table~\ref{tab1}, we show
\[
    \emin=e(4,5,n),  ~~\emax=E(4,5,n), ~~ N(x)=|\cR(4,5,n,e=x)|.
\]
Floating-point numbers in the final column are estimates obtained
by a statistical method described by the first author~\cite{orderly}.
Briefly, a backtrack program for generating the graphs exhaustively
without isomorphs by adding one vertex at a time is modified to 
randomly reject a fraction of the graphs of each size. 
The large number of graphs in the 18--20 vertex range means that
accurate counts for 21-23 vertices are more difficult to obtain.
In total, $|\cR(4,5)|\approx 2.93\times 10^{19}$.

\begin{table}[ht]
  \centering
  \begin{tabular}{c|cc|cccc|c}
      $n$ & \!$\emin$\! & \!$\emax$\! & \!$N(\emin)$\! & \!$N(\emin{+}1)$\! & \!$N(\emax{-}1)$\! & \!$N(\emax)$\! & $|\cR(4,5,n)|$ \\
      \hline
1 & 0 & 0 & 1 & 0 & 0 & 1 & 1 \\
2 & 0 & 1 & 1 & 1 & 1 & 1 & 2 \\
3 & 0 & 3 & 1 & 1 & 1 & 1 & 4 \\  
4 & 0 & 5 & 1 & 1 & 2 & 1 & 10 \\
5 & 1 & 8 & 1 & 2 & 3 & 1 & 28 \\
6 & 2 & 12 & 1 & 4 & 2 & 1 & 114 \\
7 & 3 & 16 & 1 & 3 & 4 & 1 & 627 \\
8 & 4 & 21 & 1 & 2 & 4 & 1 &  5\,588 \\
9 & 6 & 27 & 1 & 4 & 2 & 1 & 81\,321 \\
10 & 8 & 33 & 1 & 5 & 3 & 1 & 1\,915\,582 \\
11 & 10 & 40 & 1 & 3 & 2 & 1 & 67\,445\,833 \\
12 & 12 & 48 & 1 & 1 & 1 & 1 & 3\,215\,449\,959 \\
13 & 17 & 53 & 1 & 6 & 10 & 2 & 184\,701\,427\,544 \\
14 & 22 & 60 & 2 & 16 & 4 & 1 & $1.113 \times 10^{13}$ \\
15 & 27 & 66 & 1 & 6 & 14 & 1 & $5.960\times 10^{14}$ \\
16 & 32 & 72 & 1 & 2 & 138 & 5 & $2.320\times 10^{16}$ \\
17 & 41 & 79 & 5 & 226 & 86 & 1 & $5.176\times 10^{17}$ \\
18 & 50 & 85 & 326 & 29\,160 & 12\,374 & 74 & $4.98\times 10^{18}$ \\
19 & 57 & 92 & 1 & 1 & 46\,277 & 210 & $1.47\times 10^{19}$ \\
20 & 68 & 100 & 2\,016 & 213\,961 & 822 & 1 & $ 8.57\times 10^{18}$ \\
21 & 77 & 107 & 83 & 5\,940 & 10\,188 & 31 & $5.6\times 10^{17}$ \\
22 & 88 & 114 & 3 & 94 & 30\,976 & 133 & $1.8\times 10^{15}$ \\
23 & 101 & 122 & 1 & 76 & 119 & 2 & $9\times 10^{10}$ \\
24 & 116 & 132 & 9 & 90 & 3 & 2 &  352\,366 \\  
  \end{tabular}
  \smallskip
  \caption{Parameters and counts for $\cR(4,5)$\label{tab1}}
\end{table}

\bibliographystyle{plain}

\end{document}